\documentclass[11pt,reqno]{amsart}
\usepackage{amsmath}
\usepackage{amsthm}
\usepackage{amssymb}
\usepackage{enumerate}
\usepackage{amscd}
\usepackage{color}
\usepackage{pb-diagram}
\usepackage{graphicx}
\usepackage[all, cmtip]{xy}
\usepackage{soul}
\usepackage{esint}
\usepackage{hyperref}
\hypersetup{pdftex,colorlinks=true,allcolors=blue}
\usepackage{hypcap}
\usepackage[titletoc]{appendix}

\theoremstyle{plain}
\newtheorem{lemma}{Lemma}
\newtheorem{prop}[lemma]{Proposition}
\newtheorem{coro}[lemma]{Corollary}
\newtheorem{theo}[lemma]{Theorem}
\newtheorem{rema}[lemma]{Remark}
\newtheorem{thm-Intro}{Theorem} 
\newtheorem{cor-Intro}{Corollary} 

\numberwithin{equation}{section}

\newcommand{\Ric}{\textup{Ric}}

\newcommand{\Hol}{\textup{Hol}}

\begin{document}
\title[Invariant metrics on the complex ellipsoid]{Invariant metrics on the complex ellipsoid}

\author{Gunhee Cho}
\address{Department of Mathematics\\
University of Connecticut\\
196 Auditorium Road,
Storrs, CT 06269-3009, USA}

\email{gunhee.cho@uconn.edu}


\begin{abstract} 
We provide a class of geometric convex domains on which the Carath\'eodory-Reiffen metric, the Bergman metric, the complete K\"ahler-Einstein metric of negative scalar curvature are uniformly equivalent, but not proportional to each other. In a two-dimensional case, we provide a full description of curvature tensors of the Bergman metric on the weakly pseudoconvex boundary point and show that invariant metrics are proportional to each other if and only if the geometric convex domain is the Poincar\'e-disk. 
\end{abstract}

\maketitle

\tableofcontents

\section{Introduction and results}
In this paper, we study the invariant metrics on complex ellipsoid $E=E(m,n,p)=\{(z,w)\in {\mathbb{C}^n}\times{\mathbb{C}^m} : |z|^2+|w|^{2p}<1 \}$ with $p>0$.

On the unit disk $B_n$ in $\mathbb{C}^n$, the Poincar\'e-metric is the primary example for invariant metrics since invariant metrics on unit disk are just the Poincar\'e-metric up to some constant. Precisely, let's denote by $\gamma_{B_n}$, $\chi_{B_n}$, $g^B_{B_n}$, $g^{KE}_{B_n}$ the Carath\'eodory-Reiffen metric, the Kobayashi-Royden metric, the Bergman metric, and the complete K\"ahler-Einstein metric respectively. These metrics are all invariant under biholomorphisms and we have 
\begin{equation}\label{key}
\gamma_{B_n}(a;v)=\chi_{B_n}(a;v)<\sqrt{g^B_{B_n}((a;v),(a;v))}=\sqrt{g^{KE}_{B_n}((a;v),(a;v))}=c\gamma_{B_n}(a;v)
\end{equation}
for any non-zero tangent vector $(a;v)$ where $c=c(n)>0$. 
Hence if instead of the unit disk in $\mathbb{C}^n$ one considers more general complex manifolds, the four metrics provide characterizations into several classes. For a class of pseudoconvex domains in $\mathbb{C}^n$, one should expect that the relation among intrinsic metrics is different from \eqref{key} but some common phenomenon can be captured. 

With invariant metrics on complex manifolds, there is one long-standing open problem in complex geometry. That is, prove that the Carath\'eodory-Reiffen metric on a simply-connected complete K\"ahler manifold $(M, \omega)$ with negative Riemannian sectional curvature range is equivalent to other invariant metrics and related progress on this problem have been made. Recently, D. Wu and S.T. Yau showed that for this class of K\"ahler manifolds $(M, \omega)$, the base K\"ahler metric $\omega$ is uniformly equivalent to the Bergman metric, the complete K\"ahler-Einstein metric, and the Kobayashi metric (see \cite{DY17}). Based on this result, it is reasonable to ask whether such $(M, \omega)$ must be biholomorphic to a $C^k$-bounded strictly pseudoconvex-domain in $\mathbb{C}^n$ with reasonable $k\geq 0$, since it is known that the metrics are uniformly equivalent to each other for these domains (see, for example, [\cite{MBCFRG83}, \cite{CY80}, \cite{KD70}, \cite{IG75}, \cite{Lem81}, \cite{HW93}] and references therein). In this paper, we show that the complex ellipsoid serves as a bounded weakly pseudoconvex domain (with a further restriction $p>1$ from our results) in $\mathbb{C}^n$ which is not the strictly pseudoconvex domain, but those invariant metrics are uniformly equivalent and they are not proportional to each other.
 
Let's denote the complete K\"ahler-Einstein metric of the Ricci eigenvalue $-1$, the Bergman metric and the Kobayashi-Royden metric on $E=E(m,n,p)$ by $g^{KE}_E$, $g^{B}_E$ and $\chi_{E}$ respectively. Here are our results:

\begin{theo}\label{th:DE}
	Let $E=E(m,n,p)$ for any $m,n\in \mathbb{N}$ with $p>1/2$. Then there exists $C>0$ such that $g^{KE}_E$, $g^{B}_E$ and $\chi_{E}$ are uniformly equivalent to each other by $C>0$. 
\end{theo} 

Moreover, the proof of Theorem~\ref{th:DE} yields the equivalence of two invariant metrics on closed submanifolds of $E$:
\begin{coro}\label{cor:from_main}
	Let $E=E(m,n,p)$ for any $m,n\in \mathbb{N}$ with $p>1/2$. Then for any closed complex submanifold $S$ in $E$, there exists $C>0$ such that $g^{KE}_S$ and $\chi_{S}$ are uniformly equivalent by $C>0$. 
\end{coro}  

For the comparison of the K\"ahler-Einstein metric and the Bergman metric, S. Fu and B. Wong showed that for a simply-connected strictly pseudoconvex domain in $\mathbb{C}^2$ with smooth boundary, if the Bergman metric is K\"ahler-Einstein, then the domain must be biholomorphic to the disk (see \cite{SFBW97}). This is also claimed to be true for $\mathbb{C}^n$ (see, for example, \cite{XHMX16}). Since the complex ellipsoid $E=E(m,n,p)$ with $p>1$ is a weakly pseudoconvex domain because of the special boundary points $|z|=1$, further investigation should be made to compare the K\"ahler-Einstein metric with the Bergman metric on $E=E(1,1,p)$. Here is our result:

\begin{theo}\label{th:BE}
	Let $E=E(1,1,p)=\{(z,w)\in {\mathbb{C}^1}\times{\mathbb{C}^1} : |z|^2+|w|^{2p}<1 \}$ with $p>0$. Then $g^{B}_{E}\neq \lambda g^{KE}_E$ for any $\lambda>0$ if and only if $p\neq 1$.
\end{theo}

For the proof of Theorem~\ref{th:DE} and Theorem~\ref{th:BE}, the explicit formula of Bergman kernel on $E$ is used which was obtained by J.P. D'Angelo (See \cite{JPDA94}). We also provide the explicit formula of the Bergman metric and curvature tensors for Theorem~\ref{th:BE} in Section 4. 

For the comparison of Carath\'eodory-Reiffen metric and the K\"ahler-Einstein metric on $E$, we have the following: 

\begin{theo}\label{th:CE}
	Let $E=E(m,n,p)$ for any $m,n\in \mathbb{N}$ with $p>1/2$. Denote $\gamma_{E}$ be the Carath\'eodory-Reiffen metric. Then 
 \begin{equation}\label{eq:SY1}
\gamma_{E}\lneq \sqrt{{g^{KE}_E}}.
\end{equation}	
Furthermore, if we exclude $p=1$, then for any $\lambda>0$,
\begin{equation*}
\gamma_{E}\neq \lambda \sqrt{{g^{KE}_E}}.
\end{equation*}

\end{theo} 

By combining the Theorems ~\ref{th:DE}, ~\ref{th:BE},~\ref{th:CE}, we obtain:

\begin{coro}
On the complex ellipsoid $E=E(1,1,p)=\{(z,w)\in {\mathbb{C}^1}\times{\mathbb{C}^1} : |z|^2+|w|^{2p}<1 \}$, with $1\neq p>1/2$. there exists $C>0$ such that the followings hold: for any $\lambda>0$,
\begin{align*}
\frac{1}{C}\sqrt{g^{B}_{E}}&<\chi_{E}<C\sqrt{g^{B}_{E}}, \\
\frac{1}{C}\sqrt{g^{KE}_{E}}&<\chi_{E}<C\sqrt{g^{KE}_{E}}, \\
\frac{1}{C}{g^{KE}_{E}}&<g^{B}_{E}<C{g^{KE}_{E}},\\
\gamma_{E}&=\chi_{E}, \\
g^{B}_{E}& \neq \lambda g^{KE}_{E}, \\
\gamma_{E}& \neq \lambda \sqrt{{g^{KE}_E}}, \\
\gamma_{E}& \neq \lambda \sqrt{{g^{B}_E}}, \\
\gamma_{E}&\lneq \sqrt{{g^{KE}_E}}, 
\end{align*}
\begin{equation}
\gamma_{E} \lneq \sqrt{g^{B}_{E}} \label{eq:CE}. 
\end{equation}
\end{coro}

The geometric convexity of $E$ when $p> 1/2$ implies that the Carath\'eodory-Reiffen metric and the Kobayashi-Royden metric are the same (see \cite{Lem81}). Also It is known that \eqref{eq:CE} holds for any bounded domains in $\mathbb{C}^n$ (e.g, see \cite{JP13}). 

Since it is known that the Riemannian sectional curvature of the K\"ahler-Einstein metric on $E=E(1,1,p)$ is negatively pinched (see \cite{JSB86}), we have the following corollary which is related to the long-standing problem: 
\begin{coro}
There exists a simply connected, weakly (but not strictly) pseudoconvex domain in $\mathbb{C}^2$ with negative Riemannian sectional curvature range with respect to the K\"ahler-Einstein metric such that the Bergman metric, the K\"ahler-Einstein metric, the Kobayashi-Royden metric are uniformly equivalent but those are not proportional to each other. 
\end{coro}

\section{Prelinimaries}
In this section, we collect the necessary definitions that we use to prove our results. 

Let $G$ be a domain in $\mathbb{C}^n$. A pseudometric $F(z,u) : G \times \mathbb{C}^n \rightarrow [0,\infty]$ on a domain $G$ in $\mathbb{C}^n$ is called (biholomorphically) invariant if $F(z,\lambda u)=|\lambda|F(z,u)$ for all $\lambda \in \mathbb{C}^n$, and $F(z,u)=F(f(z),f’(z)u)$ for any biholomorphism $f : G \rightarrow G’$. The Caratheodory-Reiffen metric, Kobayashi-Royden metric, Bergman, K\"ahler-Einstein metric of negative scalar curvature are examples of invariant metrics on bounded weakly pseudoconvex domains in $\mathbb{C}^n$. 

Let $\mathbb{D}$ denote the open unit disk in $\mathbb{C}$. Let $z\in G$ and $v\in T_{z}G$ a tangent vector at $z$. Define the Carath\'eodory-Reiffen metric by
\[\gamma_{G}(z;v)=\sup\{|df(z)v| : f\in \Hol(G,\mathbb{D})  \}. \]
The Kobayashi-Royden metric is defined by 
\begin{equation}\label{eq:9}
\chi_{G}(z;v)=\inf\{\frac{1}{\alpha} : \alpha>0, f\in \Hol(\mathbb{D},G), f(0)=z, f'(0)=\alpha v\}.
\end{equation}

Let $\rho_{\mathbb{D}}(a,b)$ denotes the distance between two points $a,b\in \mathbb{D}$ with respect to the Poinc\'are metric of constant holomorphic sectional curvature $-4$.

The Carath\'eodory pseudo-distance $c_{G}$ on $G$ is defined by 
\[c_{G}(x,y):=\sup_{f\in \Hol(G,\mathbb{D})} \rho_{\mathbb{D}}(f(x),f(y)).\]
Here, $\rho_{\mathbb{D}}(a,b)$ denotes the integrated Poincar\'e-distance on the unit disk $\mathbb{D}$.

The Kobayashi pseudo-distance $k_{G}$ on $G$ is defined by 
\[k_{G}(x,y):=\inf_{f_i \in \Hol(\mathbb{D},G)}\left\{\sum_{i=1}^{n} \rho_{\mathbb{D}}(a_i,b_i) \right\} \]
where $x=p_0, ... ,p_n =y, f_i(a_i)=p_{i-1}, f_i(b_i)=p_i $.

The inner-Carath\'eodory length and the Kobayashi length of a piecewise $C^1$ curve $\sigma : [0,1] \rightarrow G$ are given by 
\[l^{c}(\sigma):=\int_{0}^{1}\gamma_{G}(\sigma,\sigma')dt,  \]
and
\[l^{k}(\sigma):=\int_{0}^{1}\chi_{G}(\sigma,\sigma')dt  \]
respectively. 
The inner-Carath\'eodory pseudo-distance and the inner-Kobayashi pseudo-distance on $G$ are defined by 
\[c^{i}_{G}(x,y):=\inf\{l^{c}(\sigma)(x,y) \} \text{ and } k^{i}_{G}(x,y):=\inf\{l^{k}(\sigma)(x,y) \}, \]
where the infimums are taken over all piece-wise $C^1$ curves in $G$ joining $x$ and $y$. 

The following relation is true in general: 
\begin{equation}\label{eq:10}
0\leq c_{{G}}\leq c^{i}_{G} \leq k^{i}_{G} = k_{G}. 
\end{equation}

Note that if $G$ is a bounded domain, then $k_{G}$, $c^{i}_{G}$, $c_{{G}}$ are non-degenerate and the topology induced by theses distances is the Euclidean topology. 

For a bounded domain $G$ in $\mathbb{C}^n$, denote $A^2(G)$ the holomorphic functions in $L^2(G)$. Let $\{\varphi_j : j \in \mathbb{N}  \}$ be an orthonormal basis for $A^2(G)$ with respect to the $L^2$-inner product. The Bergman kernel $K_G$ associated to $G$ is given by 
\begin{equation*}
K_{G}(z,\overline{z})=\sum_{j=1}^{\infty}\varphi_j(z)\overline{\varphi_j(z)}. 
\end{equation*}
Note that $K_{G}$ does not depend on the choice of orthonormal basis, gives rise to an invariant metric, the Bergman metric on $G$ as follows: 
\begin{equation}\label{eq:BGM}
g^B_{G}(\xi,\xi)={\sum_{\alpha,\beta =1}^{n}\frac{{\partial}^2 \log K_{G}(z,\overline{z}) }{\partial z_{\alpha}\partial \overline{z_{\beta}}}\xi_{\alpha}\overline{\xi_{\beta}}}.
\end{equation}

We say a domain $G \in \mathbb{C}^n$ is weakly pseudoconvex if $G$ has a continuous plurisubharmonic exhaustion function. In particular, every geometric convex set is a weakly pseudoconvex domain.

The existence of the complete K\"ahler-Einstein metric on a bounded pseudoconvex domain was given in the main theorem in \cite{NMY83}. Based on this result, we can always find the unique complete K\"ahler-Einstein metric of the Ricci curvature $-1$ on a bounded weakly pseudoconvex domain $G$. i.e., $g^{KE}_{G}$ satisfies $g^{KE}_{G}=-Ric_{g^{KE}_{G}}$ as a two tensor. 

\section{Equivalence of invariant metrics on ellipsoids}

To show the equivalence of Kobayashi-Royden metric, the K\"ahler-Einstein metric and the Bergman metric on $E=E(m,n,p)$, W. Yin's complete invariant K\"ahler metric $Y$ will be used. Precisely, $Y$ is the complete invariant K\"ahler metric $Y$ on $E=E(m,n,p)$ generated by the potential function
\begin{equation*}
K((z,w),\overline{(z,w)})=(1-X)^{-\lambda}(1-|z|^2)^{-N}
\end{equation*}
on $E=E(m,n,p)$, where $X=X(z,w)=|w|^2(1-|z|^2)^{-1/p}$, $N_1=(n+1)p+m$, $N=N_1/p$. Here, we will take $\lambda$ as $\lambda\geq \max\{m_1,m_2\}$, where 
\begin{align*}
m_1=&\max_{0\leq X<1}\{ \frac{F'(X)(1-X)}{F(X)} \}, \\
m_2=&\max_{0\leq X<1}\{ \frac{[F(X)F''(X)-{F(X)}^2](1-X)^2}{{F(X)}^2} \}.
\end{align*}
Then $Y$ satisfies 
\begin{equation*}
Y \geq g^{B}_{E}
\end{equation*}
(see \cite{WY97} for details). 
From  \cite{WY97}, it was shown that there exists $C>0$ such that the holomorphic sectional curvature of $Y$ on $E$ is bounded above by $-C$. Then by the generalized Schwarz lemma, there exists $C'>0$ satisfying
\begin{equation*}
\sqrt{Y(v,v)} \leq C' \chi_{E}(v)
\end{equation*}
for any vector $v \in TE$. Here, $C'$ can be taken by $\sqrt{\frac{2}{C}}$ (see the Lemma 19 in  \cite{DY17} or \cite{Y78}).
Consequently, with Lempert's classical theorem on convex domains (see \cite{Lem81}), 
\begin{equation}\label{eq:hsc1}
\sqrt{g^{B}_{E}(v)}\leq \sqrt{Y(v,v)} \leq C' \chi_{E}(v) = C' \gamma_{E}(v) \leq C' \sqrt{g^{B}_{E}(v,v)}
\end{equation}
for any vector $v \in TE$. 
We have observed that $Y$ is uniformly equivalent to the Bergman metric and the Kobayashi-Royden metric on $E$, and It remains to establish the equivalence between the K\"ahler-Einstein metric and $Y$. Now, once we have the following Proposition~\ref{prop:equiv_metric}, Theorem 3 in \cite{DY17} implies Theorem~\ref{th:DE} for $S=E$.

\begin{prop}\label{prop:equiv_metric}
There exists $D>0$ such that the holomorphic sectional curvature of $Y$ on $E$ is bounded below by $-D$.
\end{prop}

\begin{proof}
	Note that the holomorphic sectional curvature is invariant under the biholomorphic maps and for any $(z,w)\in E$, there exists an automorphism $f$ on $E$ such that $f(z,w)=(0,w')$. Thus it suffices to compute the holomorphic sectional curvature when $z=0$. The formula of the holomorphic sectional curvature $\omega[(z,w),d(z,w)]_{z=0}$ of $Y$ is explcitly given in \cite{WY97}: for any $D>0$, 
	\begin{equation*}
	\omega[(z,w),d(z,w)]_{z=0}=-D-\frac{\omega_1[(z,w),d(z,w)]}{[p^{-1}(XW'+N_1)|dz|^2+W'|dw|^2+W''|w\overline{dw}|^2]^2},
	\end{equation*}
	where $W'=\lambda(1-X)^{-1}$, $W''=\lambda(1-X)^{-2}$,
\begin{align*}
	\omega_1[(z,w),d(z,w)]=&P^{*}_1|w\overline{dw}|^4+P^{*}_{12}|w\overline{dw}|^2|{dw}|^2+P^{*}_2|{dw}|^4 \\ 
		+&Q^{*}_1|{dz}|^2|{dw}|^2+Q^{*}_2|{wdw}|^2|{dz}|^2+R^{*}|dz|^4,
\end{align*}

and $P^{*}_1$, $P^{*}_{12}$, $P^{*}_2$, $Q^{*}_1$, $Q^{*}_2$, $R^{*}$ are explicitly given as follows:
\begin{align}
	P^{*}_1=&aN_1(1-X)^{-4}(2-DaN_1), \nonumber \\
		P^{*}_{12}=&2aN_1(1-X)^{-3}(2-DaN_1),  \nonumber\\
			P^{*}_{2}=&aN_1(1-X)^{-2}(2-DaN_1), \nonumber\\
			Q^{*}_{1}=&2p^{-1}aN_1(1-X)^{-1}[(2-DaN_1)(1-X)^{-1}-DN_1(1-a)], \nonumber \\ 
				Q^{*}_{2}=&2p^{-1}(XW'+N_1)^{-1}(1-X)^{-2}aN^2_1 [(1-X)^{-2}(2a-Da^2N_1) \label{eq:hsc3} \\ 
					+&(1-X)^{-1}[4(1-a)-2DaN_1(1-a)]-D(1-a)^2N_1], \nonumber \\ 
						R^{*}=&p^{-2}N_1a[(1-X)^{-2}(2-DaN_1)+(1-X)^{-1}2(p-1 \nonumber\\
							+&DN_1(a-1))+(1-1/a)[-2p-D(a-1)N_1]			 \nonumber
\end{align}
	with $\lambda=aN_1$. We will claim that $P^{*}_1$, $P^{*}_{12}$, $P^{*}_2$, $Q^{*}_1$, $Q^{*}_2$ and $R^{*}$ are all non-positive for some $D>0$. For simplicity, let $y:=(1-X)^{-1}$.
	
 For $P^{*}_1$, $P^{*}_{12}$ and $P^{*}_2$, take $D>2(aN_1)^{-1}$. Then $X \in [0,1)$ implies $y\in [1,\infty]$, thus $P^{*}_1$, $P^{*}_{12}$ and $P^{*}_2$ are non-positive.
	
For $Q^{*}_1$, let $f_{Q^{*}_1}(y)=(2-DaN_1)y-DN_1(1-a)$. Since $D \geq 2(aN_1)^{-1}$, $f_{Q^{*}_1}(y)$ is decreasing which has the maximum $f_{Q^{*}_1}(1)$ on $[1,\infty]$. Then $f_{Q^{*}_1}(1)\leq 0$ is guaranted by taking $D \geq 2(N_1)^{-1}$.
	
For $Q^{*}_2$, let $f_{Q^{*}_2}(y)=(2a-Da^2N_1)y^2+2(1-a)(2-DaN_1)y-D(1-a^2)N_1$ from \eqref{eq:hsc3}. If $D>2(aN_1)^{-1}$, then $f_{Q^{*}_2}(y)$ has a maximum value $f_{Q^{*}_2}(y_0)=-2a^{-1}(a-1)^2<0$, where $y_0=1-1/a$.
	
	Finally, for $R^{*}$, consider $f_{R^{*}}(y)=(2-DaN_1)y^2+2(p-1+DN_1(a-1))y+(1-1/a)(-2p-D(a-1)N_1)$. Then $f_{R^{*}}'(y)=2(2-DaN_1)y+2(p-1+DaN_1-DN_1)$, thus $f_{R^{*}}'(y)$ is decreasing on $[1,\infty]$ if $D \geq 2(aN_1)^{-1}$. $f_{R^{*}}'(1)=2(p+1-DN_1)\leq 0$ if $D\geq (p+1)/N_1$, and $f_{R^{*}}(1)=(1/a)(2p-DN_1)\leq 0$ if $D \geq (2p/N_1)$. Hence if $D\geq \max\{ 2(aN_1)^{-1}, (p+1)/N_1,(2p/N_1)\}$ then $f_{R^{*}}(y)\leq 0$. 	
	
	In all, if $D\geq \max\{ 2(aN_1)^{-1}, 2(N_1)^{-1}, (p+1)/N_1,(2p/N_1)\}$, then $P^{*}_1$, $P^{*}_{12}$, $P^{*}_2$, $Q^{*}_1$, $Q^{*}_2$ and $R^{*}$ are all non-positive. Hence we obtain 
	\begin{equation*}
	\omega[(z,w),d(z,w)]_{z=0} \geq -D.
	\end{equation*}
\end{proof}
\begin{rema}
From the proof of Theorem~\ref{th:DE} with Theorem 3 in \cite{DY17}, $C>0$ only depends on the negative holomorphic sectional curvature range of the W. Yin's complete invariant K\"ahler metric in \cite{WY97}. In the special case of $E=E(m,1,p)$ with $p \geq 1$, we can also use the negative Riemannian sectional curvature range of the K\"ahler-Einstein metric on $E$ to determine $C>0$ (see Theorem 4 in \cite{JSB86}). 
\end{rema}
\begin{rema}
Theorem~\ref{th:DE} can be proved by an alternative approach, which combines Theorem 1 in \cite{KKLZ16} and Theorem 7.2 in \cite{KLXSSTY04} or Theorem 2 in \cite{SKY09}. However, our approach has the further consequence which holds to any closed complex submanifold $S$ in $E=E(m,n,p)$: Since we can restrict W. Yin's complete invariant K\"ahler metric $Y$ to $S$, which still has the negative pinched holomorphic sectional curvature range on $S$. Then by Theorem 2 and Theorem 3 in \cite{DY17}, we have Corollary~\ref{cor:from_main}.
\end{rema}

\section{Bergman metric and its curvatures for two-dimensional ellipsoids}
In this section, we will investigate the two-dimensional complex ellipsoid $E=E(1,1,p)=\{(z,w)\in {\mathbb{C}^1}\times{\mathbb{C}^1} : |z|^2+|w|^{2p}<1 \},p>0$. In order to prove Theorem~\ref{th:BE}, we will provide a detailed description of curvature tensors of the Bergman metric near to the special boundary points $|z|=1$ on $E$. Recall that with the global coordinate $(z,w)\in E$ in $\mathbb{C}^{2}$, let $\{\frac{\partial}{\partial{z}},\frac{\partial}{\partial{w}} \}$ be the basis on $T^{1,0}_{0}E$.
In \cite{JPDA94}, the formula of the Bergman kernel $K_{E}$ on $E=E(1,1,p)=\{(z,w)\in {\mathbb{C}^1}\times{\mathbb{C}^1} : |z|^2+|w|^{2p}<1 \}$ is explicitly known: 
\begin{equation}\label{eq:EBK}
K_{E}((z,w),\overline{(z,w)})=c_1 \frac{(1-|z|^2)^{-2+\frac{1}{p}}}{((1-|z|^2)^{1/p}-|w|^2)^2}+c_2 \frac{(1-|z|^2)^{-2+\frac{2}{p}}}{((1-|z|^2)^{1/p}-|w|^2)^3},
\end{equation}
where $c_2$, $c_1$ are given by 
\begin{equation*}
c_{1}=\frac{1}{{p\pi}^{2}}(p-1), c_{2}=\frac{2}{{p\pi}^{2}}.
\end{equation*}
For computations, define $\phi(z,\overline{z},w,\overline{w})=1-z\overline{z}$ and  $\psi(z,\overline{z},w,\overline{w})=(1-z\overline{z})^{\frac{1}{p}}-w\overline{w}$. Then $K_E=c_1 \phi^{-2+\frac{1}{p}}\psi^{-2}+c_2\phi^{-2+\frac{2}{p}}\psi^{-3}$. 
For convenience, let's establish the notation. We will denote $\frac{\partial}{\partial z}$ by $\partial_{1}$, and $\frac{\partial}{\partial \overline{z}}=:\partial_{\overline{1}}$, $\frac{\partial}{\partial \overline{w}}=:\partial_{\overline{2}}$, $\frac{\partial}{\partial \overline{w}}=:\partial_{\overline{2}}$. 
From \eqref{eq:BGM}, the metric component $g^{B}_{i\overline{j}}$ of $g^{B}_{E}$ is given by 
\begin{equation}\label{eq:metric}
g^{B}_{i\overline{j}}=\frac{{\partial}^2 \log K_{E}(z,\overline{z}) }{\partial z_{i}\partial \overline{z_{j}}}=K^{-2}_{E}(K_{E}{{\partial^2_{i\overline{j}}}  K_{E}}-{{\partial_i}  K_{E}}{\partial _{\overline{j}}}{ K_{E}}), i=1,2.
\end{equation}
Hence with \eqref{eq:EBK}, the elementary computation gives the following proposition.
\begin{prop}\label{prop:metric}
The components of $g^{B}_{E}$ for any $(z,w)\in E$ are given as follows:
\begin{align*}
g_{1\overline{1}}&=a_1 a_2 a_3, \\
g_{1\overline{2}}&=a_2a_4 a_5, \\
g_{2\overline{1}}&=a_2 a_4 a_6, \\
g_{2\overline{2}}&=a_2 a_4 a_7, 
\end{align*}
where each $a_{i\overline{j},k}$ is a function of $(z,\overline{z},w,\overline{w})$ as below:
\begin{align*}
a_1=&2 p^4 \psi^4+5 p^3 \psi^3 (\phi^{\frac{1}{p}}+w\overline{w})\\
&+2 p^2\psi^2 (w\overline{w} (z\overline{z}+5) \phi^{\frac{1}{p}}+2\phi^{\frac{2}{p}}+2 (w\overline{w})^2) \\
&+4 z\overline{z}w\overline{w} \phi^{\frac{1}{p}} (w\overline{w} \phi^{\frac{1}{p}}+\phi^{\frac{2}{p}}+(w\overline{w})^2)\\
&+p (\phi^{\frac{2}{p}}-(w\overline{w})^2)(2 w\overline{w} (3 z\overline{z}+2) \phi^{\frac{1}{p}}+\phi^{\frac{2}{p}}+(w\overline{w})^2),  \\
a_2=&\psi^{-2} (p \psi+\phi^{\frac{1}{p}}+w\overline{w})^{-2}, \\
a_3=&p^{-2} \phi^{-2}, \\
a_4=&p^2 \psi^2+3p(\phi^{\frac{2}{p}}-(w\overline{w})^2)+2(w\overline{w} \phi^{\frac{1}{p}}+\phi^{\frac{2}{p}}+(w\overline{w})^2), \\
a_5=&2p^{-1}\overline{z}w \phi^{-1+\frac{1}{p}}, \\
a_6=&2p^{-1}\overline{w}z \phi^{-1+\frac{1}{p}}, \\ 
a_7=&2 \phi^{\frac{1}{p}}.
\end{align*}
\begin{proof}
Notice that all of the first-order derivatives and second-order derivatives of $K_E$ are written as linear combinations of $h\phi^{\alpha}\psi^{\beta}$ with some functions $h=h(z,\overline{z},w,\overline{w})$ and $\alpha,\beta \in \mathbb{R}$. We will provide each term of those in order to proceed the computation: the first-order derivatives of $\phi$ and $\psi$ are given as follows:
\begin{align*}
\partial_1 \phi&=-\overline{z}, \partial_{\overline{1}} \phi=-z, \partial_{2} \phi=\partial_{\overline{2}} \phi=0, \\
\partial_{1} \psi&=-\frac{1}{p}\overline{z}\phi^{\frac{1}{p}-1}, \partial_{\overline{1}} \psi=-\frac{1}{p}{z}\phi^{\frac{1}{p}-1}, \partial_{2} \psi=-\overline{w}, \partial_{\overline{2}}\psi=-w.
\end{align*}
Then by direct computations, the first-order derivatives of $K_E$ are written as:
\begin{align*}
\partial_{1} K_E=&\frac{(p-1) (2 p-1)}{\pi ^2 p^2}\overline{z}\phi^{-3+\frac{1}{p}}\psi^{-2}+\frac{2 (3 p-5)}{\pi ^2 p^2}\overline{z}\phi^{-3+\frac{2}{p}}\psi^{-3}+\frac{6}{\pi ^2 p^2}\overline{z}\phi^{-3+\frac{3}{p}}\psi^{-4}, \\
\partial_{\overline{1}} K_E=&\frac{(p-1) (2 p-1)}{\pi ^2 p^2}{z}\phi^{-3+\frac{1}{p}}\psi^{-2}+\frac{2 (3 p-5)}{\pi ^2 p^2}{z}\phi^{-3+\frac{2}{p}}\psi^{-3}+\frac{6}{\pi ^2 p^2}{z}\phi^{-3+\frac{3}{p}}\psi^{-4}, \\ 
\partial_{2} K_E=&\frac{2(p-1)}{p\pi^2}\overline{w}\phi^{-2+\frac{1}{p}}\psi^{-3}+\frac{6}{p\pi^2}\overline{w}\phi^{-2+\frac{2}{p}}\psi^{-4},\\
\partial_{\overline{2}} K_E=&\frac{2(p-1)}{p\pi^2}{w}\phi^{-2+\frac{1}{p}}\psi^{-3}+\frac{6}{p\pi^2}{w}\phi^{-2+\frac{2}{p}}\psi^{-4}.
\end{align*}
The second-order derivatives of $K_E$ are written as:
\begin{align*}
\partial^2_{1\overline{1}}K_E=&\frac{(p-1) (2 p-1)}{\pi ^2 p^2}\phi^{-3+\frac{1}{p}}\psi^{-2}+\frac{6 (p-1)}{\pi ^2 p^2}\phi^{-3+\frac{2}{p}}\psi^{-3}+\frac{6}{\pi ^2 p^2}\phi^{-3+\frac{3}{p}}\psi^{-4} \\
&+\frac{(p-1) (2 p-1) (3 p-1)}{\pi ^2 p^3}\phi^{-4+\frac{1}{p}}\psi^{-2}+\frac{2 (p-1) (11 p-7)}{\pi ^2 p^3}z\overline{z}\phi^{-4+\frac{2}{p}}\psi^{-3} \\
&+\frac{6 (p-1) (3 p+1)}{\pi ^2 p^3}z\overline{z}\phi^{-4+\frac{3}{p}}\psi^{-4}+\frac{24}{\pi ^2 p^3}z\overline{z}\phi^{-4+\frac{4}{p}}\psi^{-5},\\
\partial^2_{1\overline{2}}K_E=&\frac{4 p^2-6 p+2}{\pi ^2 p^2}w\overline{z}\phi^{-3+\frac{1}{p}}\psi^{-3}+\frac{18 (p-1)}{\pi ^2 p^2}w\overline{z}\phi^{-3+\frac{2}{p}}\psi^{-4}+\frac{24}{\pi ^2 p^2}{p}w\overline{z}\phi^{-3+\frac{3}{p}}\psi^{-5}, \\ 
\partial^2_{2\overline{1}}K_E=&\frac{4 p^2-6 p+2}{\pi ^2 p^2}z\overline{w}\phi^{-3+\frac{1}{p}}\psi^{-3}+\frac{18 (p-1)}{\pi ^2 p^2}z\overline{w}\phi^{-3+\frac{2}{p}}\psi^{-4}+\frac{24}{\pi ^2 p^2}z\overline{w}\phi^{-3+\frac{3}{p}}\psi^{-5}, \\ 
\partial^2_{2\overline{2}}K_E=&\frac{2(p-1)}{p\pi^2}\phi^{-2+\frac{1}{p}}\psi^{-3}+\frac{6(p-1)}{p\pi^2} w\overline{w}\phi^{-2+\frac{1}{p}}\psi^{-4}+\frac{6}{p\pi^2}\phi^{-2+\frac{2}{p}}\psi^{-4}+\frac{24}{p\pi^2}w\overline{w}\phi^{-2+\frac{2}{p}}\psi^{-5}.
\end{align*}

Notice that $\phi^{-4+\frac{2}{p}}\psi^{-6}$ becomes the common factor to the numerator and the denominator of $\frac{K_{E}{{\partial^2_{i\overline{j}}}  K_{E}}-{{\partial_i}  K_{E}}{\partial _{\overline{j}}}{ K_{E}}}{{K_{E}}^2}$ for indices  $(i,\overline{j})=(1,\overline{1}),(1,\overline{2}),(2,\overline{1}),(2,\overline{2})$. Thus by multiplying $\psi^{2}$ on both sides after canceling the common factor $\phi^{-4+\frac{2}{p}}\psi^{-6}$, we get the common denominator-term $a_2$ for any $g_{i\overline{j}}$. Other two terms $a_k,a_l$ of $g_{i\overline{j}}=a_2a_k a_l$ can be obtained from direct computation. For the rest of the proof, we proceed the concrete computation for $g_{1\overline{1}}$ and other metric components follow from similar computations. 
From \eqref{eq:metric},
\begin{align*}
g_{1\overline{1}}&=\frac{K_{E}{{\partial^2_{1\overline{1}}}  K_{E}}-{{\partial_1} K_{E}}{\partial _{\overline{1}}}{ K_{E}}}{{K_{E}}^2}=\frac{{\phi}^{-4+\frac{2}{p}}{\psi}^{-6} {\widehat{(K_{E}{{\partial^2_{1\overline{1}}}  K_{E}}}-\widehat{{{\partial_1} K_{E}}{\partial _{\overline{1}}}{ K_{E}})}} }{{\phi}^{-4+\frac{2}{p}}{\psi}^{-6}(\widehat{{K_{E}}^2)}}\\
&=\frac{{\psi^2 \widehat{(K_{E}{{\partial^2_{1\overline{1}}}  K_{E}}}-\widehat{{{\partial_1} K_{E}}{\partial _{\overline{1}}}{ K_{E}})}} }{\psi^2 (\widehat{{K_{E}}^2)}},
\end{align*}
where 
\begin{equation*}
\psi^2 (\widehat{{K_{E}}^2)}=\psi^{2} (p \psi+\phi^{\frac{1}{p}}+w\overline{w})^{2}=\frac{1}{a_2},
\end{equation*}
and
\begin{align}
&\psi^2 ({\widehat{K_{E}{{\partial^2_{1\overline{1}}}  K_{E}}}-\widehat{{{\partial_1} K_{E}}{\partial _{\overline{1}}}{ K_{E}}}}) \label{eq:a11}  \\ 
&={\frac{(p-1)^2(2p-1)(3p-1)-(2p^2-3p+1)^2 z\overline{z} }{p^2}  }\phi^{-2}\psi^{4}+{\frac{(2p^2-p)(p-1)^2}{p^2}}\phi^{-1}\psi^{4} \nonumber \\
&{\frac{-2(p-1)(p^2-8p+3)z\overline{z}+2(p-1)(2p-1)(3p-1)}{p^2} }\phi^{-2+\frac{1}{p}}\psi^{3}+\frac{10p^3-18p^2+8p}{p^2}\phi^{-1+\frac{1}{p}}\psi^{3}  \nonumber\\
&+{\frac{18(p^2-p)}{p^2}}\phi^{-1+\frac{2}{p}}\psi^{2}+{\frac{18p^3-46p^2+90p+78}{p^2}}z\overline{z}\phi^{-2+\frac{2}{p}}\psi^{2} \nonumber\\
&+{\frac{12(3p^2-6p+7)}{p^2}}z\overline{z}\phi^{-2+\frac{3}{p}}\psi+{\frac{12}{p}}\phi^{-1+\frac{3}{p}}\psi+{\frac{12}{p^2}}z\overline{z}\phi^{-2+\frac{4}{p}} \nonumber.
\end{align}
Take $a_3=\frac{1}{p^2\phi^2}$, then one can see that with $z\overline{z}=1-\phi$, the term of the highest degree of $\psi$ of the rest terms of \eqref{eq:a11} becomes $2p^4\psi^4$ and further simplication with $\psi=\phi^{1/p}-w\overline{w}$ yields $a_1$.
\end{proof}
\end{prop}

From this proposition, we can observe that two vector fields $\frac{\partial}{\partial z}$ and $\frac{\partial}{\partial \overline{w}}$ are orthogonal to each other with respect to the Bergman metric $g^{B}_E$ if $a_5=a_6=0$. Hence it is reasonable to evaluate the curvature tensors with the choice of points $(z,w)\in E$ satisfying
\begin{equation*}
z=\overline{z} \text{ and } w=0.
\end{equation*}
or 
\begin{equation*}
w=\overline{w} \text{ and } z=0.
\end{equation*}
 Since it's interesting to know those tensor components on $C^2$ weakly pseudoconvex boundary points $|z|=1$ in the case $p>1$, we proceed the computation with the former case.

First, let us compute the components of Ricci curvature tensor of the Bergman metric $g^{B}_{E}$. From the well-known formula of the Ricci curvature tensor with the K\"ahler metric $g^{B}_{E}$, 
\begin{align}
\text{Ric}_{i\overline{j}}(g^{B}_{E})&=-{{\partial_{i\overline{j}}}  \log \det{g^{B}_{E}}} \nonumber \\
&=(\det{g^{B}_{E}})^{-2}({{\partial_i} \det{g^{B}_{E}}}{{\partial_{\overline{j}}} \det{g^{B}_{E}}}-(\det{g^{B}_{E}}){{\partial_{i\overline{j}}}  \det{g^{B}_{E}}}). \label{eq:Ric}
\end{align}

By previous propositions, many terms in the components of Ricci curvatures of $g^{B}_{E}$ are vanished, so that we can compute to obtain the following propositions:

\begin{prop}\label{prop:Ric}
On the point $(z,w)\in E$ satisfying $z=\overline{z},w=0$, the components of the Ricci tensor of $g^{B}_{E}$ are given as follows:
\begin{align*}
	\Ric_{1\overline{1}}&=-\frac{2p+1}{p \left(1-z\overline{z}\right)^2}, \\
		\Ric_{1\overline{2}}&=\Ric_{2\overline{1}}=0, \\
			\Ric_{2\overline{2}}&=-\frac{2 \left(2 p^3+10 p^2+10 p+5\right)}{(p+1)
			(p+2) (2 p+1)\left(1-z\overline{z}\right)^{1/p}}.
\end{align*}
\begin{proof}
From \eqref{eq:Ric}, we should establish the formulas of the zero, first, and second-order derivatives of $\det{g^{B}_{E}}$. To do so, from the formula
\begin{equation*}
\det{g^{B}_{E}}=g_{1\overline{1}}g_{2\overline{2}}-g_{1\overline{2}}g_{2\overline{1}}=a_1(a_2)^2a_3a_4a_7-{a_2}^2{a_4}^2a_5a_6,
\end{equation*}
it is necessary to determine all formulas of zero, first, and second-order derivatives of $a_i,i=1,...,7$ which are given in Proposition~\ref{prop:metric}. 	
For the zero-order derivatives of $a_i$'s, by putting $z=\overline{z}, w=0$ in formulas in Proposition~\ref{prop:metric}, we have:
\begin{align*}
a_1&=p (p+1)^2 (2 p+1) \left(1-z\overline{z}\right)^{4/p}, \\
a_2&=\frac{1}{(p+1)^2\left(1-z\overline{z}\right)^{4/p}}, a_3=\frac{1}{p^2 \left(1-z\overline{z}\right)^2},
\\
a_4&=\left(p^2+3 p+2\right) \left(1-z\overline{z}\right)^{2/p}, a_7=2 \left(1-z\overline{z}\right)^{\frac{1}{p}}, \\ 
a_i&=0, i=5,6
\end{align*}
Then by using $a_5=a_6=0$, 
\begin{equation*}
\det{g^{B}_{E}}=a_1(a_2)^2a_3a_4a_7=\frac{2 (2 p+1) \left(p^2+3 p+2\right) \left(1-z\overline{z}\right)^{-1/p}}{p (p+1)^2
	\left(z\overline{z}-1\right)^2}.
\end{equation*}
Next, we compute the first-order derivatives of $a_i$'s, then substituting $z=\overline{z},w=0$ yields the following:
\begin{align*}
\partial_1 a_{1}&=\partial_{\overline{1}} a_{1}=-4 (p+1)^2 (2 p+1) z\left(1-z\overline{z}\right)^{\frac{4}{p}-1}, \\
\partial_1 a_{2}&=\partial_{\overline{1}} a_{2}=\frac{4 z \left(1-z\overline{z}\right)^{-\frac{4}{p}-1}}{p (p+1)^2}, \\ 
\partial_1 a_{3}&=\partial_{\overline{1}} a_{3}=\frac{2 z}{p^2 \left(1-z\overline{z}\right)^3}, \\
\partial_1 a_{4}&=\partial_{\overline{1}} a_{4}=-\frac{2 \left(p^2+3 p+2\right) z \left(1-z\overline{z}\right)^{\frac{2}{p}-1}}{p}, \\
\partial_2 a_{5}&=\partial_{\overline{2}} a_{6}=\frac{2 z \left(1-z\overline{z}\right)^{\frac{1}{p}-1}}{p},\\
\partial_1 a_{7}&=\partial_{\overline{1}} a_{7}=-\frac{2 z \left(1-z\overline{z}\right)^{\frac{1}{p}-1}}{p},\\
\partial_2 a_{1}&=\partial_{\overline{2}} a_{1}=\partial_2 a_{2}=\partial_{\overline{2}} a_{2}=\partial_2 a_{3}=\partial_{\overline{2}} a_{3}=\partial_2 a_{4}=\partial_{\overline{2}} a_{4} \\
&=\partial_1 a_{5}=\partial_{\overline{1}} a_{5}=\partial_{\overline{2}} a_{5}=\partial_1 a_{6}=\partial_{\overline{1}} a_{6}=\partial_2 a_{6}=\partial_2 a_{7}=\partial_{\overline{2}} a_{7}=0.
\end{align*}
In particular, from those vanishing terms
\begin{equation*}
\partial_{{i}}a_1=\partial_{{i}}a_2=\partial_{{i}}a_3=\partial_{{i}}a_4=\partial_{{i}}a_7=0,i=2,\overline{2},
\end{equation*}
and $a_5=a_6=0$, we have 
\begin{equation*}
\partial_2\det{g^{B}_{E}}=\partial_{\overline{2}}\det{g^{B}_{E}}=0.
\end{equation*}
On the other hand, from $\partial_1 a_{i}=\partial_{\overline{1}} a_{i},i=1,2,3,4$ with zero, and first-order derivatives of $a_i$'s, computation yields
\begin{equation*}
\partial_1\det{g^{B}_{E}}=\partial_{\overline{1}}\det{g^{B}_{E}}=\frac{2 (p+2) (2 p+1)^2 z\left(1-z\overline{z}\right)^{-\frac{1}{p}-3}}{p^2
	(p+1)}.
\end{equation*}
Now, we compute the second-order derivatives of $a_i$'s, then substituting $z=\overline{z},w=0$ yields the following:
\begin{align*}
\partial^2_{1\overline{1}}a_{1}&=-\frac{4 (p+1)^2 (2 p+1) \left(p-4 z\overline{z}\right)\left(1-z\overline{z}\right)^{\frac{4}{p}-2}}{p}, \\
\partial^2_{2\overline{2}} a_{1}&=-2 (p+1) \left(1-z\overline{z}\right)^{3/p} \left(4 p^3+p^2-p \left(z\overline{z}+2\right)-2z\overline{z}\right), \\
\partial^2_{1\overline{1}}a_{2}&=\frac{4 \left(1-z\overline{z}\right)^{-\frac{2 (p+2)}{p}} \left(p+4 z\overline{z}\right)}{p^2(p+1)^2}, \\
\partial^2_{2\overline{2}} a_{2}&=\frac{4 p \left(1-z\overline{z}\right)^{-5/p}}{(p+1)^3}, \\
\partial^2_{1\overline{1}}a_{3}&=\frac{4 z\overline{z}+2}{p^2 \left(z\overline{z}-1\right)^4}, \\
\partial^2_{1\overline{1}}a_{4}&=-\frac{2 \left(p^2+3 p+2\right) \left(p-2 z\overline{z}\right)
	\left(1-z\overline{z}\right)^{\frac{2}{p}-2}}{p^2}, \\ 
\partial^2_{2\overline{2}} a_{4}&=-2 \left(p^2-1\right) \left(1-z\overline{z}\right)^{\frac{1}{p}}, \\
\partial^2_{2\overline{1}} a_{5}&=\partial^2_{1\overline{2}} a_{6}=\frac{2 \left(1-z\overline{z}\right)^{\frac{1}{p}-2} \left(p-z\overline{z}\right)}{p^2}, \\ 
\partial^2_{1\overline{1}}a_{7}&=\frac{2 \left(1-z\overline{z}\right)^{\frac{1}{p}-2} \left(z\overline{z}-p\right)}{p^2}, \\
\partial^2_{1\overline{2}} a_{1}&=\partial^2_{2\overline{1}} a_{1}=\partial^2_{1\overline{2}} a_{2}=\partial^2_{2\overline{1}} a_{2}=\partial^2_{1\overline{2}} a_{3}=\partial^2_{2\overline{1}} a_{3}=\partial^2_{2\overline{2}} a_{3}=\partial^2_{1\overline{2}} a_{4}=\partial^2_{2\overline{1}} a_{4} \\
&=\partial^2_{1\overline{1}}a_{5}=\partial^2_{1\overline{2}} a_{5}=\partial^2_{2\overline{2}} a_{5}=\partial^2_{1\overline{1}}a_{6}=\partial^2_{2\overline{1}} a_{6}=\partial^2_{2\overline{2}} a_{6}=\partial^2_{1\overline{2}} a_{7}=\partial^2_{2\overline{1}} a_{7}=\partial^2_{2\overline{2}} a_{7}=0.
\end{align*}
For two terms $\partial^2_{1\overline{2}}\det{g^{B}_{E}}$ and $\partial^2_{2\overline{1}}\det{g^{B}_{E}}$, one can check the following: expand $\partial^2_{1\overline{2}}\det{g^{B}_{E}}$ and $\partial^2_{2\overline{1}}\det{g^{B}_{E}}$ as linear combinations of $a_i$'s and derivatives of those. Then each term in expressions contains at least one zero term. Consequently, we have 
\begin{equation*}
\partial^2_{1\overline{2}}\det{g^{B}_{E}}=\partial^2_{2\overline{1}}\det{g^{B}_{E}}=0.
\end{equation*}
For the other two terms, direct computation with second-order derivatives of $a_i$'s yields:
\begin{align*}
\partial^2_{1\overline{1}}\det{g^{B}_{E}}&=\frac{2 (p+2) (2 p+1)^2 \left(1-z\overline{z}\right)^{-\frac{1}{p}-4} \left(2 p
	z\overline{z}+p+z\overline{z}\right)}{p^3 (p+1)}, \\
\partial^2_{2\overline{2}}\det{g^{B}_{E}}&=\frac{\left(8 p^3+40 p^2+40 p+20\right) \left(1-z\overline{z}\right)^{-\frac{2
			(p+1)}{p}}}{p (p+1)^2}. 
\end{align*}
By combining with \eqref{eq:Ric}, we have all desired formulas from direct computations. 
\end{proof}	
\end{prop}

From the next Proposition, we have Theorem~\ref{th:BE} as a Corollary. Moreover, not as in the smoothly bounded strictly pseudoconvex domain case, the Bergman metric even fails to be asymptotically K\"ahler-Einstein on $E$ (also, see \cite{CY80}).

\begin{prop}\label{prop:aysmp-Ric} 
On the point $(z,w)\in E$ satisfying $z=\overline{z},w=0$, we have 
\begin{equation*}
\frac{\Ric_{1\overline{1}}}{g_{1\overline{1}}}=-1,
\end{equation*}
\begin{equation*}
\frac{\Ric_{2\overline{2}}}{g_{2\overline{2}}}=-\frac{2 p^3+10 p^2+10 p+5}{(p+2)^2 (2 p+1)},
\end{equation*}
In particular, $\frac{\Ric_{1\overline{1}}}{g_{1\overline{1}}}=\frac{\Ric_{2\overline{2}}}{g_{1\overline{1}}}$ if and only if $p=1$.
\begin{proof}
From the formulas of $a_i$'s which are obtained in the proof in Proposition~\ref{prop:Ric} at $(z,\overline{z},0,0)\in E$ with $z=\overline{z}$, we have metric components
\begin{align*}
g_{1\overline{1}}&=a_1a_2a_3=\frac{2p+1}{p \left(1-z\overline{z}\right)^2}, \\
g_{2\overline{2}}&=a_2a_4a_7=\frac{2(p+2)}{(p+1)\left(1-z\overline{z}\right)^{1/p}}.
\end{align*}
Then by combining the formulas of Ricci curvature in Proposition~\ref{prop:Ric}, the result follows from the direct computation. 
\end{proof}

\end{prop}

The holomorphic sectional curvature $h$ on $l$-dimensional complex hermitian manifold $(M,g)$ in the holomorphic tangent vector $\xi=\sum_{i=1}^{l}\xi_i \frac{\partial}{\partial z_i}$ is given by 
\begin{equation*}
h(\xi)=\frac{2R_{\xi \overline{\xi} \xi \overline{\xi}}}{{g(\xi,\overline{\xi})}^2}=\frac{ {\tiny }2\sum_{a,b,c,d=1}^{l}R_{a\overline{b}c\overline{d}}\xi_a \overline{\xi}_b \xi_c \overline{\xi}_d }{ \sum_{a,b,c,d=1}^{l}g_{a\overline{b}}g_{c\overline{d}}  \xi_a \overline{\xi}_b \xi_c \overline{\xi}_d},
\end{equation*}
where the components of curvature tensor $R$ associated with $g$ is given by 
\begin{equation}\label{eq:curv}
R_{a\overline{b}c\overline{d}}=-\frac{\partial^2 g_{a\overline{b}}}{\partial z_c \partial \overline{z}_d}+\sum_{p,q=1}^{l}g^{q\overline{p}}\frac{\partial g_{a\overline{p}}}{\partial z_c}\frac{\partial g_{q\overline{b}}}{\partial \overline{z}_d}.
\end{equation}

\begin{prop}\label{prop:hsc}
On the point $(z,w)\in E$ satisfying $z=\overline{z},w=0$, the components of the holomorphic sectional curvatures $h$ of $g^{B}_{E}$ are given as follows:
\begin{align*}
h({\frac{\partial}{\partial z}})&=-\frac{2p}{1+2p}, \\
h({\frac{\partial}{\partial w}})&=-\frac{1+4p+p^2}{(2+p)^2}.
\end{align*}

In particular, $h({\frac{\partial}{\partial z}})=h({\frac{\partial}{\partial w}})$ if and only if $p=1$.	
\begin{proof}	
From the formulas of $a_i, \partial_{j}a_i$'s which are obtained in the proof in Proposition~\ref{prop:Ric},
$g_{1\overline{2}}=g_{2\overline{1}}=0$ because of $a_4=a_5=0$. Also, from $\partial_{1} a_5=\partial_{\overline{1}} a_6=\partial_{2} a_5=\partial_{\overline{2}} a_6=0$, we have 
\begin{equation*}
\partial_{1}g_{1\overline{2}}=\partial_{\overline{1}}g_{2\overline{1}}=\partial_{2}g_{1\overline{2}}=\partial_{\overline{2}}g_{2\overline{1}}=0.
\end{equation*}
Hence the components of the curvature tensor in \eqref{eq:curv} become
\begin{align*}
R_{1\overline{1}1\overline{1}}&=-\partial^2_{1\overline{1}}g_{1\overline{1}}+g^{1\overline{1}}\partial_{1}g_{1\overline{1}}\partial_{\overline{1}}g_{1\overline{1}}, \\
R_{2\overline{2}2\overline{2}}&=-\partial^2_{2\overline{2}}g_{2\overline{2}}+g^{2\overline{2}}\partial_{2}g_{2\overline{2}}\partial_{\overline{2}}g_{2\overline{2}}.
\end{align*}
From the formulas of $a_i$'s which are obtained in the proof in Proposition~\ref{prop:Ric} and Proposition~\ref{prop:aysmp-Ric},
\begin{align*}
g^{1\overline{1}}&=\frac{g_{2\overline{2}}}{\det g^{B}_E}=\frac{p \left(1-z\overline{z} \right)^2}{2 p+1}, \\
g^{2\overline{2}}&=\frac{g_{1\overline{1}}}{\det g^{B}_E}=\frac{(p+1) \left(1-z\overline{z}\right)^{\frac{1}{p}}}{2 (p+2)}.
\end{align*}
Combining with necessary formulas of $\partial_{j}a_i, \partial^2_{i\overline{j}}a_{k}$'s in the proof in Proposition~\ref{prop:Ric}, with the formulas of the inverse metrics given above, the result follows from the direct computation. 
\end{proof}
\end{prop}

Proposition~\ref{prop:hsc} yields the following consequence:
\begin{coro}
the Bergman metric is not proportional to the Kobayashi-Royden metric on $E=E(1,1,p)$ with $1\neq p>1/2$. 
\begin{proof}
Suppose the Bergman metric is proportional to the Kobayashi-Royden metric. i.e., $\chi_{E}=\lambda\sqrt{g^{B}_{E}}$ for some $\lambda>0$. From the geometric-convexity of $E$, the Carath\'eodory-Reiffen metric is the same as the Kobayashi-Royden metric. Then the holomorphic sectional curvature of the Bergman metric must be the constant by applying Theorem 1 in \cite{BW77}, which is only possible when $p=1$ from Proposition~\ref{prop:hsc}. 
\end{proof}
\end{coro}

\section{The Carath\'eodory-Reiffen metric on geometric convex domains}

In the following proposition, we consider bounded geometric convex domains $\Omega$ in $\mathbb{C}^n$ and distinguish the Carath\'eodory-Reiffen metric from K\"ahler-Einstein metric. Then Proposition~\ref{CKE-comparison} directly implies Theorem~\ref{th:CE}.

\begin{prop}\label{CKE-comparison}
For any bounded geometric convex domain $\Omega$ in $\mathbb{C}^n$, let $g^{KE}_{\Omega}$ be the complete K\"ahler-Einstein metric of Ricci curvature $-1$. Then we have
\begin{equation*}
\gamma_{\Omega}(a;v)\leq \sqrt{{g^{KE}_\Omega}((a;v),(a;v))} \text{  } \text{for all nonzero tangent vectors $(a;v)$},
\end{equation*}
and
\begin{equation*}
\gamma_{\Omega}(a;v)< \sqrt{{g^{KE}_\Omega}((a;v),(a;v))} \text{  } \text{for some nonzero tangent vector $(a;v)$}.
\end{equation*}	
Furthermore, if $\Omega$ is not biholomorphic to the unit disk in $\mathbb{C}^n$, then for any $\lambda>0$, 
\begin{equation*}
\gamma_{\Omega}(a;v) \neq \lambda \sqrt{{g^{KE}_\Omega}((a;v),(a;v))} \text{  } \text{for some nonzero tangent vector $(a;v)$}.
\end{equation*}	

\begin{proof}
	Notice that the first inequality is the consequence of the generalized Schwarz lemma (see \cite{Y78}). To show the second inequality, suppose 
	\begin{equation*}
	\gamma_{\Omega}(a;v)= \sqrt{{g^{KE}_\Omega}((a;v),(a;v))} \text{  } \text{for all nonzero tangent vector $(a;v)$}.
	\end{equation*}
	Then we have  
	\begin{equation*}
	c^{i}_{\Omega}=d^{{KE}}_{\Omega}.
	\end{equation*}
	Since $\Omega$ is geometric convex, the Carath\'eodory-Reifen metric is same as the Kobayashi-Royden metric. In particular, this forces Carath\'eodory pseudo-distance must be inner. Then from \eqref{eq:10}, $c_{{\Omega}}= c^{i}_{\Omega}=d^{{KE}}_{\Omega}$, which contradicts the main theorem in \cite{GC18}.
	In the case that a bounded geometric convex domain $\Omega$ is not biholomorphic to the disk,  if we assume further that $\gamma_{\Omega}(a;.)=\lambda \sqrt{{g^{KE}_\Omega}(.,.)}$, for some $\lambda>0$, then $\Omega$ must be biholomorphic to the unit disk by Theorem 2 in \cite{BW77}, which is impossible. 
\end{proof}
\end{prop}

\subsection*{Acknowledgements}
This work was partially supported by NSF grant DMS-1611745. I would like to thank my advisor Professor Damin Wu for many insights and deep encouragements. Also, I appreciate Prof. Yunhui Wu who pointed out the alternative proof of the Theorem~\ref{th:DE} after submitting the paper on ArXiv. I appreciate to Hyun Chul Jang who is the Ph.D. candidate in Uconn for helpful discussions. 
	

\end{document}